\documentclass[12pt]{amsart}

\usepackage{amsmath}
\usepackage{amssymb}
\usepackage{amsfonts}
\usepackage{amsthm}
\usepackage{enumerate}
\usepackage{hyperref}
\usepackage{color}

\theoremstyle{plain}
\newtheorem{thm}{Theorem}[section]

\newtheorem{prop}[thm]{Proposition}
\newtheorem{lem}[thm]{Lemma}
\newtheorem{cor}[thm]{Corollary}

\theoremstyle{definition}
\newtheorem{dfn}[thm]{Definition}

\newtheorem{exmps}[thm]{Examples}
\newtheorem{rem}[thm]{Remark}

\newtheorem{dfns-rems}[thm]{Definitions and Remarks}
\newtheorem{notas-rems}[thm]{Notations and Remarks}
\newtheorem{exmps-rems}[thm]{Examples and Remarks}

\begin{document}


\title[On the Stanley depth of squarefree monomial ideals]{On the Stanley depth of squarefree monomial ideals}


\author[S. A. Seyed Fakhari]{S. A. Seyed Fakhari}

\address{S. A. Seyed Fakhari, School of Mathematics, Institute for Research
in Fundamental Sciences (IPM), P.O. Box 19395-5746, Tehran, Iran.}

\email{fakhari@ipm.ir}

\urladdr{http://math.ipm.ac.ir/fakhari/}


\begin{abstract}
Let $\mathbb{K}$ be a field and $S=\mathbb{K}[x_1,\dots,x_n]$ be the
polynomial ring in $n$ variables over the field $\mathbb{K}$. Suppose that $\mathcal{C}$ is a chordal clutter with $n$ vertices and assume that the minimum edge cardinality of $\mathcal{C}$ is at least $d$. It is shown that $S/I(c_d(\mathcal{C}))$ satisfies Stanley's conjecture, where $I(c_d(\mathcal{C}))$ is the edge ideal of the $d$-complement of $\mathcal{C}$. This, in particular shows that $S/I$ satisfies Stanley's conjecture, where $I$ is a quadratic monomial ideal with linear resolution. We also define the notion of Schmitt--Vogel number of a monomial ideal $I$, denoted by ${\rm sv}(I)$ and prove that for every squarefree monomial ideal $I$, the inequalities ${\rm sdepth}(I)\geq n-{\rm sv}(I)+1$ and ${\rm sdepth}(S/I)\geq n-{\rm sv}(I)$ hold.
\end{abstract}


\subjclass[2000]{Primary: 13C15, 05E40; Secondary: 05C65, 13C13}


\keywords{Stanley depth, Chordal clutter, Schmitt--Vogel number}


\thanks{This research was in part supported by a grant from IPM (No. 93130422)}


\maketitle


\section{Introduction} \label{sec1}

Let $\mathbb{K}$ be a field and $S=\mathbb{K}[x_1,\dots,x_n]$ be the
polynomial ring in $n$ variables over the field $\mathbb{K}$. Let $M$ be a nonzero
finitely generated $\mathbb{Z}^n$-graded $S$-module. Let $u\in M$ be a
homogeneous element and $Z\subseteq \{x_1,\dots,x_n\}$. The $\mathbb
{K}$-subspace $u\mathbb{K}[Z]$ generated by all elements $uv$ with $v\in
\mathbb{K}[Z]$ is called a {\it Stanley space} of dimension $|Z|$, if it is
a free $\mathbb{K}[\mathbb{Z}]$-module. Here, as usual, $|Z|$ denotes the
number of elements of $Z$. A decomposition $\mathcal{D}$ of $M$ as a finite
direct sum of Stanley spaces is called a {\it Stanley decomposition} of
$M$. The minimum dimension of a Stanley space in $\mathcal{D}$ is called the
{\it Stanley depth} of $\mathcal{D}$ and is denoted by ${\rm
sdepth}(\mathcal {D})$. The quantity $${\rm sdepth}(M):=\max\big\{{\rm
sdepth}(\mathcal{D})\mid \mathcal{D}\ {\rm is\ a\ Stanley\ decomposition\
of}\ M\big\}$$ is called the {\it Stanley depth} of $M$. Stanley \cite{s}
conjectured that $${\rm depth}(M) \leq {\rm sdepth}(M)$$ for all
$\mathbb{Z}^n$-graded $S$-modules $M$. For a reader friendly introduction
to Stanley decomposition, we refer to \cite{psty} and for a nice survey on this topic we refer to \cite{h}.

It is shown in \cite[Corollary 4.5]{ikm} that in order to prove Stanley's Conjecture for the modules of the form $I/J$, where $J\subset I$ are monomial ideals, it is enough to consider the case when $I$ and $J$ are squarefree monomial ideals. Thus, in this paper, we restrict ourselves to squarefree monomial ideals.

In Section \ref{sec2}, we consider a class of monomial ideals with linear quotients. By a result of Fr${\rm \ddot{o}}$berg, we know that a quadratic squarefree monomial ideal ideal has linear resolution if and only if it is the edge ideal of a graph with chordal complement. Herzog, Hibi and Zheng \cite{hhz} prove that this is equivalent to say that $I$ has linear quotients. In \cite{w}, Woodroofe extends the definition of chordal to clutters (see Definition \ref{defchord}). Assume that $\mathcal{C}$ is chordal clutter and suppose that the minimum edge cardinality of $\mathcal{C}$ is at least $d$, for a fixed integer $d\geq 1$. Woodroofe proves that the edge ideal $I(c_d(\mathcal{C}))$ of the $d$-complement of $\mathcal{C}$ (defined in Section \ref{sec2}) has linear quotients. Using this result, in Theorem \ref{main}, we prove that $S/I(c_d(\mathcal{C}))$ satisfies Stanley's conjecture. As a consequence, we conclude that $S/I$ satisfies Stanley's conjecture, if $I$ is a quadratic (not necessarily squarefree) monomial ideal with linear resolution (see Corollary \ref{linres}).

In Section \ref{sec3}, we provide a lower bound for the Stanley depth of squarefree monomial ideals. In fact, for every monomial ideal $I$, we introduce the notion of Schmitt--Vogel number (see Definition \ref{sv}), denoted by ${\rm sv}(I)$ and prove that for every squarefree monomial ideal $I$, the inequalities ${\rm sdepth}(I)\geq n-{\rm sv}(I)+1$ and ${\rm sdepth}(S/I)\geq n-{\rm sv}(I)$ hold (see Theorem \ref{smain}).


\section{Chordal clutters and Stanley's conjecture} \label{sec2}

In this section, we prove the first main result of this paper (Theorem  \ref{main}). Before starting the proof, we introduce some notation and well known facts.

If $I\subseteq S$ is a squarefree monomial ideal, we
can identify the set of minimal monomial generators of $I$ with the edge set of a clutter, defined as follows.

\begin{dfn}
Let $V$ be a finite set. A {\it clutter} $\mathcal{C}$ with vertex set $V$ consists of a set of
subsets of $V$, called the {\it edges} of $\mathcal{C}$, with the property that no edge contains another.
\end{dfn}

We write $V(\mathcal{C})$ to denote the vertex set of $\mathcal{C}$, and $E(\mathcal{C})$ to denote its edge set. Let $\mathcal{C}$ be a clutter and assume that $V(\mathcal{C})=\{v_1, \ldots, v_n\}$. For every subset $e\subseteq \{v_1, \ldots, v_n\}$, we write $x_e$ to denote the squarefree monomial $\prod_{v_i \in e} x_i$. Then the {\it edge ideal} of $\mathcal{C}$ is defined to be
$$I(\mathcal{C}) = (x_e : e\in E(\mathcal{C})),$$as an ideal in the polynomial ring $S = \mathbb{K}[x_1, \dots, x_n]$.

In the following definition, we mention two types of operations preformed on a clutter $\mathcal{C}$
to produce smaller clutters.
\begin{dfn}
Let $\mathcal{C}$ be a clutter and $v$ be a vertex of $\mathcal{C}$.
\begin{itemize}
\item[(i)] The {\it deletion} $\mathcal{C}\setminus \{v\}$ is the clutter with vertex set $V(\mathcal{C})\setminus\{v\}$ and edge set $E(\mathcal{C}\setminus v)=\{e\in E(\mathcal{C}): v\notin e\}$.

\item[(ii)] The {\it contraction} $\mathcal{C}/\{v\}$ is the clutter with vertex set $V(\mathcal{C})\setminus\{v\}$ whose edges are the minimal elements of the set $\{e\setminus \{v\}: e\in E(\mathcal{C})\}$.
\end{itemize}
A clutter obtained from $\mathcal{C}$ by applying a sequence of deletion and/or contraction is called a {\it minor} of $\mathcal{C}$.
\end{dfn}

Let $G$ be a graph. For a vertex $v\in V(G)$, the neighborhood of $v$ in $G$ is defined to be the set $N_G(v)=\{w\in V(G) : \{v,w\}\in E(G)\}$. A graph is called chordal if every cycle of length at least four has a chord. We recall that a chord of a cycle is an edge which joins two vertices of the cycle but is not itself an edge of the cycle. By Dirac's theorem \cite{D}, the graph $G$ is chordal if and only if every induced subgraph of $G$ has a simplicial vertex, i.e., a vertex whose neighborhood forms a complete subgraph of $G$. Woodroofe \cite{w} extends the concept of chordalness to clutters.

\begin{dfn} \label{defchord}
Let $\mathcal{C}$ be a clutter. A vertex $v\in V(\mathcal{C})$ is called a {\it simplicial vertex} if for every two distinct edges $e_1, e_2 \in E(\mathcal{C})$ containing $v$, there exists an edge $e_3 \in E(\mathcal{C})$ with $e_3\subseteq (e_1\cup e_2)\setminus \{v\}$. The clutter $\mathcal{C}$ is called {\it chordal} if every minor of $\mathcal{C}$ has a simplicial vertex.
\end{dfn}

For a clutter $\mathcal{C}$ and a fixed integer $d\geq 1$, Woodroofe \cite{w} defines the {\it $d$-complement} of $\mathcal{C}$, denoted by $c_d(\mathcal{C})$, to ba a clutter with the vertex set $V(\mathcal{C})$ and the edge set $\{e \subseteq V(\mathcal{C}): \ \mid e\mid = d \ {\rm and} \ e\notin E(\mathcal{C})\}$.

Let $I$ be a monomial ideal and let $G(I)$ be the set of minimal monomial generators of $I$. Assume that $u_1\succ u_2 \succ \ldots \succ u_t$ is a linear order on $G(I)$. We say that $I$ has {\it linear quotients with respect to $\succ$}, if for every $2\leq i\leq t$, the ideal $(u_1, \ldots, u_{i-1}):u_i$ is generated by a subset of variables. We say that $I$ has {\it linear quotients}, if it has linear quotients with respect to a linear order on $G(I)$.

Let $\mathcal{C}$ be a chordal clutter and assume that the minimum edge cardinality of $\mathcal{C}$ is at least $d$. In \cite{w}, Woodroofe proves that the edge ideal of $c_d(\mathcal{C})$ has linear quotients. Woodroofe's proof is based on some results regarding the Alexander duality of squarefree monomial ideals. Here we restate the proof for illustrating more information about the particular linear order on $G(I(c_d(\mathcal{C})))$ which satisfies the conditions of having linear quotients. Then we use these information to prove that Stanley's conjecture holds for $S/I(c_d(\mathcal{C}))$.

{\bf Notation.} Let $\mathcal{C}$ be a clutter with the vertex set $\{v_1, \ldots, v_n\}$. For every squarefree monomial $u\in S=\mathbb{K}[x_1,\dots,x_n]$ we set ${\mathbf e}(u)=\{v_i\in V(\mathcal{C}): x_i \ {\rm divides} \ u\}$.

\begin{prop} \label{linquo}
Let $\mathcal{C}$ be a chordal clutter and assume that the minimum edge cardinality of $\mathcal{C}$ is at least $d$. Suppose that $v_i$ is a simplicial vertex of $\mathcal{C}$. There exists a linear order $\succ$ on $G(I(c_d(\mathcal{C})))$ such that $I(c_d(\mathcal{C}))$ has linear quotients with respect to $\succ$ and moreover, $u\succ v$ for every pair of monomials $u, v\in G(I(c_d(\mathcal{C})))$ with $x_i\mid u$ and $x_i\nmid v$.
\end{prop}
\begin{proof}
Set $I=I(c_d(\mathcal{C}))$. There is nothing to prove if $n=1$ or $d=1$. Thus assume that $n\geq 2$ and $d\geq 2$. We prove the assertion by induction on the number of vertices $n$. Note that $(I: x_i)$ is the edge ideal of $c_d(\mathcal{C})/\{v_i\}$. By \cite[Lemma 6.7]{w}, $c_d(\mathcal{C})/\{v_i\}$ is the $d-1$-complement of $\mathcal{C}/\{v_i\}$. This shows that $(I: x_i)$ is generated in a single degree, namely $d-1$ and thus $$G((I:x_i))=\{\frac{u}{x_i} : u\in G(I) \ {\rm and} \ x_i\mid u\}.$$On the other hand $\mathcal{C}/\{v_i\}$ is a chordal clutter and therefore, using the induction hypothesis, we conclude that $(I: x_i)$ has linear quotients. This shows that the ideal $(u\in G(I): x_i\mid u)$ has linear quotients. Thus, there exists a linear order $u_1\succ u_2 \succ \ldots \succ u_t$ on $G((I: x_i))$ such that for every $2\leq j\leq t$, the ideal $(u_1, \ldots, u_{j-1}):u_j$ is generated by a subset of variables.

The ideal $I':=(u\in G(I): x_i\nmid u)$ is the edge ideal of $c_d(\mathcal{C})\setminus \{v\}=c_d(\mathcal{C}\setminus \{v\})$. Since $\mathcal{C}\setminus \{v\}$ is a chordal clutter, the induction hypothesis implies that the ideal $(u\in G(I): x_i\nmid u)$ has linear quotients. Thus, there is a linear order $w_1\succ w_2 \succ \ldots \succ w_s$ on $G(I')$ such that for every $2\leq j\leq s$, the ideal $(w_1, \ldots, w_{j-1}):w_j$ is generated by a subset of variables.

We claim that $I$ has linear quotients with respect to the following order on $G(I)$ $$u_1\succ u_2 \succ \ldots \succ u_t\succ w_1\succ w_2 \succ \ldots \succ w_s.$$Note that for every $1\leq k\leq t$ and every $1\leq m \leq s$, the variable $x_i$ divides the monomial $u_k/{\rm gcd}(u_k, w_m)$. Hence, in order to prove the claim it is enough to show that for every $1\leq m \leq s$ there exists a variable, say $x_{\ell}$ dividing $w_m$, such that $x_iw_m/x_{\ell}\in G(I)$. Assume by contradiction that for every variable $x_{\ell}$ dividing $w_m$ we have $x_iw_m/x_{\ell}\notin G(I)$. Since $d\geq 2$, this shows that there exist two distinct integers $\ell_1, \ell_2\neq i$ such that $x_{\ell_1}$ and $x_{\ell_2}$ divide $w_m$ and ${\mathbf e}(x_iw_m/x_{\ell_1})=({\mathbf e}(w_m)\setminus \{v_{\ell_1}\})\cup \{v_i\}$ and ${\mathbf e}(x_iw_m/x_{\ell_2})=({\mathbf e}(w_m)\setminus \{v_{\ell_2}\})\cup \{v_i\}$ are edges of $\mathcal{C}$. Since $v_i$ is a simplicial vertex of $\mathcal{C}$ and the minimum cardinality of edges of $\mathcal{C}$ is at least $d$, we conclude that ${\mathbf e}(w_m)$ is an edge of $\mathcal{C}$ and thus $w_m\notin G(I)$. This is a contradiction and completes the proof of the proposition.
\end{proof}

The following result due to Sharifan and Varbaro has a crucial role in the proof of our main result.

\begin{thm} [\cite{sv}, Corollary 2.7] \label{pd}
Let $I\subseteq S=\mathbb{K}[x_1,\dots,x_n]$ be a monomial ideal. Assume that $I$ has linear quotients with respect to $u_1\succ u_2 \succ \ldots \succ u_t$, where $\{u_1, \ldots, u_t\}$ is the set of minimal
monomial generators of $I$. For every $2\leq i\leq t$, let $n_i$ be the number of variables which generate $(u_1, \ldots, u_{i-1}):u_i$. Then $${\rm pd}_S(S/I)={\rm max}\{n_i: 2\leq i\leq t\}+1.$$
\end{thm}

Keeping the notations of Theorem \ref{pd} in mind, Auslander--Buchsbaum formula implies that $${\rm depth}_S(S/I)=n-{\rm max}\{n_i: 2\leq i\leq t\}-1.$$

Let $\mathcal{C}$ be a clutter with a simplicial vertex and assume the the minimum edge cardinality of $\mathcal{C}$ is at least $d$. In order to prove Stanley's conjecture for $S/I(c_d(\mathcal{C}))$, we need the following lemma. It shows that the depth of $I(c_d(\mathcal{C}))$ does not decrease under the elimination of a suitable variable. As usual for every monomial $u$, the {\it support} of $u$, denoted by ${\rm Supp}(u)$, is the set of variables which divide $u$.

\begin{lem} \label{del}
Let $\mathcal{C}$ be a clutter and assume that the minimum edge cardinality of $\mathcal{C}$ is at least $d$. Suppose that $v_i$ is a simplicial vertex of $\mathcal{C}$. Set $I=I(c_d(\mathcal{C}))$ and assume that $$x_i\in \bigcup_{u\in G(I)}{\rm Supp}(u).$$Let $S'=\mathbb{K}[x_1, \ldots, x_{i-1}, x_{i+1}, \ldots,  x_n]$ be the polynomial ring obtained from $S$ by deleting the variable $x_i$ and consider the ideal $I'=I\cap S'$. Then ${\rm depth}_{S'}(S'/I')\geq {\rm depth}_S(S/I)$.
\end{lem}
\begin{proof}
We first note that $I\neq 0$, since $$x_1\in \bigcup_{u\in G(I)}{\rm Supp}(u).$$If $I'=0$, then ${\rm depth}_{S'}(S'/I')=n-1\geq {\rm depth}_S(S/I)$. Therefore, assume that $I'\neq 0$. There is nothing to prove if $d=1$. Thus, assume that $d\geq 2$. By Proposition \ref{linquo}, there exists a linear order $\succ$ on $G(I)$ such that $I$ has linear quotients with respect to $\succ$ and moreover, $u\succ v$ for every pair of monomials $u, v\in G(I)$ with $x_i\mid u$ and $x_i\nmid v$. Let $G(I)=\{u_1, \ldots, u_m\}$ be the set of minimal monomial generators of $I$ and assume that $u_1 \succ \cdots \succ u_m$. By assumption, there exists and integer $t$ with $1\leq t \leq m$, such that $x_i$ divides $u_1, \ldots, u_t$ and does not divide $u_{t+1}, \ldots, u_m$. We claim that for every integer $k$ with $t+1\leq k \leq m$, there exists an integer $1\leq j_k \leq n$ with $j_k\neq i$ such that $(u_k/x_{j_k})x_i\in G(I)$. Indeed, assume that this is not the case. Since $d\geq 2$, there exist integers $\ell\neq \ell'$ with $(u_k/x_{\ell})x_i\notin G(I)$ and $(u_k/x_{\ell'})x_i\notin G(I)$. This means that ${\mathbf e}(u_k/x_{\ell})x_i)$ and ${\mathbf e}(u_k/x_{\ell'})x_i)$ are edges of $\mathcal{C}$ and contain the vertex $v_i$. Since $v_i$ is a simplicial vertex and since the minimum edge cardinality of $\mathcal{C}$ is at least $d$, we conclude that ${\mathbf e}(u_k)$ is an edge of $\mathcal{C}$ and hence $u_k\notin G(I)$ which is a contradiction and proves the claim. This shows that for every $k$ with $t+2\leq k \leq m$,$$(x_i)+((u_{t+1}, \ldots, u_{k_1}):u_{k})\subseteq ((u_k/x_{j_k})x_i, u_{t+1}, \ldots, u_{k-1})):u_k\subseteq (u_1, \ldots, u_{k-1}):u_k.$$On the other hand, it is clear that$$x_1\notin (u_{t+1}, \ldots, u_{k-1}):u_k$$Therefore, by Theorem \ref{pd}, we conclude that ${\rm pd}_S(S/I)\geq {\rm pd}_{S'}(S'/I')+1$. Now Auslander--Buchsbaum formula completes the proof of the Lemma.
\end{proof}

We are now ready to prove the main result of this section.

\begin{thm} \label{main}
Let $\mathcal{C}$ be a chordal clutter and assume that the minimum edge cardinality of $\mathcal{C}$ is at least $d$. Then $S/I(c_d(\mathcal{C}))$ satisfies Stanley's conjecture.
\end{thm}
\begin{proof}
Set $I=I(c_d(\mathcal{C}))$. We prove the assertion by induction on $n$, where $n$ is the number of vertices of $\mathcal{C}$. If
$n=1$, then $I$ is a principal ideal and so we have ${\rm depth}(S/I)=n-1$ and by \cite[Theorem 1.1]{r},
${\rm sdepth}(S/I)=n-1$. Therefore, in
this case, the assertion is trivial.

We now assume that $n\geq 2$. Without loss of generality assume that $v_1$ is a simplicial vertex of $\mathcal{C}$. Let $S'=\mathbb{K}[x_2, \ldots, x_n]$ be the polynomial ring obtained from $S$ by deleting the variable $x_1$ and consider the ideals $I'=I\cap S'$ and
$I''=(I:x_1)$. If $$x_1\notin \bigcup_{u\in G(I)}{\rm Supp}(u),$$then
${\rm depth}(S/I)={\rm depth}(S'/I')+1$. Also, by \cite[Lemma 3.6]{hvz}, we conclude that ${\rm sdepth}(S/I)={\rm
sdepth}(S'/I')+1$. On the other hand, in the case we have $c_d(\mathcal{C})=c_d(\mathcal{C}\setminus \{v\})$ and thus $I'$ is the edge ideal of $c_d(\mathcal{C}\setminus \{v\})$.
Therefore, using the induction hypothesis, we conclude that ${\rm sdepth}(S/I)\geq {\rm depth}(S/I)$.
Hence, we may assume that $$x_1\in \bigcup_{u\in G(I)}{\rm Supp}(u).$$
Now $S/I=(S'/I'S')\oplus x_1(S/I''S)$ and therefore by the definition of  the Stanley depth we have

\[
\begin{array}{rl}
{\rm sdepth}(S/I)\geq \min \{{\rm sdepth}_{S'}(S'/I'S'), {\rm sdepth}_S(S/I'')\}.
\end{array} \tag{1} \label{1}
\]

Using \cite[Lemma 6.7]{w}, it follows that $I''$ is the edge ideal of $c_{d-1}(\mathcal{C}/\{v\})$. Note that $\mathcal{C}/\{v\}$ is a chordal clutter with minimum edge cardinality at least $d-1$. Hence \cite[Lemma 3.6]{hvz}, \cite[Corollary 1.3]{r2} and the induction hypothesis implies that $${\rm sdepth}_S(S/I'')= {\rm sdepth}_{S'}(S'/I''S')+1 \geq {\rm depth}_{S'}(S'/I''S')+1$$
$$={\rm depth}_S(S/I'')\geq {\rm depth}_S(S/I).$$

On the other hand, $I'S'$ is the edge ideal of $c_d(\mathcal{C}\setminus \{v\})$ and since $$x_1 \in \bigcup_{u\in G(I)} {\rm Supp}(u),$$using Lemma \ref{del}, we conclude that ${\rm depth}_{S'}(S'/I')\geq {\rm depth}_S(S/I)$ and since $\mathcal{C}\setminus \{v\}$ is a chordal clutter, using
the induction hypothesis, we conclude that ${\rm sdepth}_{S'}(S'/I'S')\geq {\rm depth}_S(S/I)$. Now the assertions follow by inequality (\ref{1}).
\end{proof}

If we restrict our attention to the graphs, we obtain the following corollary. We mention that for a graph $G=(V(G),E(G))$, its {\it complementary graph} $\overline{G}$ is a graph with $V(\overline{G})=V(G)$ and $E(\overline{G})$
consists of those $2$-element subsets $\{v_i,v_j\}$ of $V(G)$ for which $\{v_i,v_j\}\notin E(G)$.

\begin{cor} \label{chord}
Let $G$ be a graph with chordal complement and $I=I(G)$ its edge ideal. Then $S/I$ satisfies Stanley's conjecture.
\end{cor}

We are now able to prove Stanley's conjecture for every quadratic (not necessarily squarefree) monomial ideal with linear resolution. We recall that a monomial ideal $I$ is said to have {\it linear resolution}, if for some integer $t$, the graded Betti numbers $\beta_{i,i+j}(I)$ vanish,
for all $i$ and every $j\neq t$.

\begin{cor} \label{linres}
Let $I$ be a quadratic monomial ideal with linear resolution. Then $S/I$ satisfies Stanley's conjecture.
\end{cor}
\begin{proof}
We use polarization (see \cite{hh'} for the definition of polarization). Let $I^p$ denote the polarization of $I$ which is considered in a new polynomial ring, say $T$. Then $I^p$ is a quadratic squarefree monomial ideal. On the other hand, it follows from \cite[Corollary 1.6.3]{hh'} that $I^p$ has linear resolution. Using Fr${\rm \ddot{o}}$berg's result \cite[Theorem 1]{f}, we conclude that $I^p$ is the edge ideal of a graph with chordal complement. Thus Corollary \ref{chord} implies that $T/I^p$ satisfies Stanley's conjecture. Now \cite[Corollary 4.5]{ikm} implies that $S/I$ satisfies Stanley's conjecture.
\end{proof}

We close this section by the following remark.

\begin{rem}
Let $\mathcal{C}$ be a chordal clutter and assume that the minimum edge cardinality of $\mathcal{C}$ is at least $d$. In Theorem \ref{main}, we showed that $S/I(c_d(\mathcal{C}))$ satisfies Stanley's conjecture. It is natural to ask whether $I(c_d(\mathcal{C}))$ itself satisfies Stanley's conjecture. The answer of this question is positive. Indeed, Soleyman Jahan \cite{so} proves that Stanley's conjecture holds true for every monomial ideal with linear quotients.
\end{rem}


\section{Schmitt--Vogel number and Stanley depth} \label{sec3}

In this section we provide a lower bound for the Stanley depth of squarefree monomial ideals. The lower bound is given in terms of the Schmitt--Vogel number which is defined in the following definition.

\begin{dfn} \label{sv}
Let $I$ be a monomial ideal and let ${\rm Mon}(I)$ be the set of monomials of $I$. The Schmitt--Vogel number of $I$, denoted by ${\rm sv}(I)$ is the smallest integer $r$ for which there exist subsets $P_1, P_2, \ldots, P_r$ of ${\rm Mon}(I)$ such that
\begin{itemize}
\item[(i)] $\mid P_1\mid =1$ and

\item [(ii)] For all $\ell$ with $1 <\ell\leq r$ and for all $u, u''\in P_{\ell}$ with $u\neq u''$, there exists an integer $\ell'$ with $1 \leq \ell' < \ell$ and an element $u'\in P_{\ell'}$ such that $uu''\in (u')$

\item[(iii)] $I$ is generated by the set $\bigcup_{i=1}^rP_i$.
\end{itemize}
\end{dfn}

\begin{rem}
Schmitt and Vogel \cite[p. 249]{sv1} prove that for every monomial ideal $I$, the quantity ${\rm sv}(I)$ is an upper bound for the arithmetical rank of $I$.
\end{rem}

Let $P$ be a monomial prime ideal in $S$, and $I\subseteq S$ any monomial ideal. We denote by $I(P)$ the monomial ideal in the polynomial ring $S(P) = \mathbb{K}[x_j: x_j\notin P]$, which is obtained from $I$ by applying the $\mathbb{K}$-algebra homomorphism $S\rightarrow S(P)$ with $x_i\mapsto 1$ for all $i\in P$. The ideal $I(P)$ is called the {\it monomial localization} of $I$ at the prime ideal $P$.

\begin{lem} \label{loc}
Let $I$ be a monomial ideal of $S$ and $P\subset S$ be a monomial prime. Then ${\rm sv}(I(P))\leq {\rm sv}(I)$.
\end{lem}

\begin{proof}
Assume that ${\rm sv}(I)=r$ and let $P_1, P_2, \ldots, P_r$ be the subsets of ${\rm Mon}(I)$ which satisfy the conditions of Definition \ref{sv}. To prove the assertion, it is enough to apply the $\mathbb{K}$-algebra homomorphism $S\rightarrow S(P)$ with $x_i\mapsto 1$ for all $i\in P$, to every set $P_j$ with $1\leq j \leq r$.
\end{proof}

Assume that $I$ is a squarefree monomial ideal and $P=(x_i)$ a principal monomial prime ideal of $S$. Then it is clear that $I(P)=(I:x_i)$. Therefore as a consequence of lemma \ref{loc} we obtain the following corollary.

\begin{cor} \label{swcolon}
Let $I$ be a squarefree monomial ideal. Then for every $1\leq i \leq n$, we have ${\rm sv}((I:x_i))\leq {\rm sv}(I)$.
\end{cor}

In the following lemma we consider the behavior of the Schmitt--Vogel number of an arbitrary monomial ideal under the elimination of variables.

\begin{lem} \label{del}
Let $I$ be a monomial ideal of $S=\mathbb{K}[x_1,\ldots,x_n]$. Then there exists a variable $x_i$ such that ${\rm sv}(I\cap S')+1\leq {\rm sv}(I)$, where $S'=\mathbb{K}[x_1, \ldots, x_{i-1}, x_{i+1}, \ldots x_n]$ is the polynomial ring obtained from $S$ by deleting the variable $x_i$.
\end{lem}

\begin{proof}
Assume that ${\rm sv}(I)=r$ and let $P_1, P_2, \ldots, P_r$ be the subsets of ${\rm Mon}(I)$ which satisfy the conditions of Definition \ref{sv}. Assume that $P_1=\{u\}$ and suppose that $x_i$ is a variable which divides $u$. Set $S'=\mathbb{K}[x_1, \ldots, x_{i-1}, x_{i+1}, \ldots x_n]$ and $P_j'=P_j\cap S'$ for every $1\leq j\leq r$. Then $P_1'=\emptyset$. Thus, there exist integers $2\leq i_1 < i_2 < \ldots < i_t\leq r$ such that $P_{i_k}'\neq \emptyset$ for every $1\leq k \leq t$ and $P_j'=\emptyset$ for every $j\notin \{i_1, i_2 \ldots, i_t\}$. It is clear that $\bigcup_{k=1}^tP_{i_k}'$ is a generating set for $I'$. Since $i_1\geq 2$, it follows that $t \leq r-1$. Hence, in order to prove the assertion, it is enough to prove that the sets $P_{i_1}', \ldots, P_{i_t}'$ satisfy conditions (i) and (ii) of Definition \ref{sv}.

We first verify condition (i). Assume that $\mid P_{i_1}'\mid\geq 2$. This means that there exist two monomials $u_1\neq u_2$ in $P_{i_1}$ which are not divisible by $x_i$. Thus, by condition (ii) of Definition \ref{sv}, there exists and integer $m < i_1$ and a monomial $u_3\in P_m$ with $u_3\mid u_1u_2$. But this is not possible. Because $P_m'=\emptyset$ and therefore, every element of $P_m$ and in particular $u_3$ is divisible by $x_i$. This proves condition (i).

To prove condition (ii), let $v_1\neq v_2$ be two monomials in $P_{i_k}'$ for some $k$ with $1 < k \leq t$. Then $v_1, v_2\in P_{i_k}$ and since $P_1, P_2, \ldots, P_r$ satisfy condition (ii) of Definition \ref{sv}, it follows that there exists and integer $s$ with $1\leq s < i_k$ and a monomial $v_3\in P_s$, such that $v_3$ divides $v_1v_2$. Since $v_1$ and $v_2$ are not divisible by $x_i$, we conclude that $x_i\nmid v_3$. Thus, $s\in \{i_1, \ldots, i_t\}$ and $v_3\in P_s'$. This verifies condition (ii) of Definition \ref{sv} and completes the proof of the lemma.
\end{proof}

We are now ready to state and prove the main result of this section.

\begin{thm} \label{smain}
Let $I$ be a squarefree monomial ideal of $S=\mathbb{K}[x_1,\ldots,x_n]$. Then ${\rm sdepth}(I)\geq n-{\rm sv}(I)+1$ and ${\rm sdepth}(S/I)\geq n-{\rm sv}(I)$.
\end{thm}

\begin{proof}
There is nothing to prove if $I=0$. Thus assume that $I\neq 0$. We prove the assertions by induction on $n$. If
$n=1$, then $I$ is a principal ideal and so we have ${\rm sv}(I)=1$, ${\rm
sdepth}(I)=n$ and by \cite[Theorem
1.1]{r}, ${\rm sdepth}(S/I)=n-1$. Therefore, in
this case, the assertions are trivial.

We now assume that $n\geq 2$. By Lemma \ref{del} there exists a variable $x_i$ such that ${\rm sv}(I\cap S')+1\leq {\rm sv}(I)$, where $S'=\mathbb{K}[x_1, \ldots, x_{i-1}, x_{i+1}, \ldots x_n]$ is the polynomial ring obtained from $S$ by deleting the variable $x_i$. Set $I''=(I:x_i)$.

Now $I=I'S'\oplus x_iI''S$ and $S/I=(S'/I'S')\oplus
x_i(S/I''S)$ and therefore by the definition of Stanley depth we have
\[
\begin{array}{rl}
{\rm sdepth}(I)\geq \min \{{\rm sdepth}_{S'}(I'S'), {\rm sdepth}_S(I'')\},
\end{array} \tag{1} \label{1}
\]
and
\[
\begin{array}{rl}
{\rm sdepth}(S/I)\geq \min \{{\rm sdepth}_{S'}(S'/I'S'), {\rm sdepth}_S(S/I'')\}.
\end{array} \tag{2} \label{2}
\]
Note that the generators of $I''$ belong to $S'$. Therefore our induction hypothesis implies that $${\rm sdepth}_{S'}(S'/I'')\geq (n-1)-{\rm sv}(I'')$$ and  $${\rm sdepth}_{S'}(I'')\geq (n-1)-{\rm sv}(I'')+1$$ Using Corollary \ref{swcolon} together with \cite[Theorem
1.1]{r} and \cite[Lemma 3.6]{hvz}, we conclude that
$${\rm sdepth}(S/I'')={\rm sdepth}_{S'}(S'/I'')+1\geq (n-1)-{\rm sv}(I'')+1\geq n-{\rm sv}(I),$$
and
$${\rm sdepth}_S(I'')={\rm sdepth}_{S'}(I'')+1\geq (n-1)-{\rm sv}(I'')+1+1\geq n-{\rm sv}(I)+1.$$

On the other hand, by the choice of $x_i$ we have ${\rm sv}(I'S')\leq {\rm sv}(I)-1$ and therefore by
 induction hypothesis we conclude that

$${\rm sdepth}_{S'}(I'S')\geq (n-1)-{\rm sv}(I'S')+1\geq (n-1)-({\rm sv}(I)-1)+1$$ $$=n-{\rm sv}(I)+1,$$
and similarly ${\rm sdepth}_{S'}(S'/I'S')\geq n-{\rm sv}(I)$. Now the assertions follow by inequalities (\ref{1}) and (\ref{2}).
\end{proof}

In \cite{ks}, the authors determine two lower bounds for the Stanley depth of monomial ideals (see \cite[Corollary 2.5 and Theorem 3.2]{ks}). In the following examples, we show that these bounds are not stronger than the bound given in Theorem \ref{smain}.

\begin{exmps}
\begin{enumerate}
\item Consider the ideal $I = (xy, xz, yzt) \subset S = \mathbb{K}[x,y,z,t]$.
	It is easy to see that the lcm number of $I$ (see \cite[Definition 1.1]{ks}) is equal to $3$. Thus, Corollary \cite[Corollary 2.5]{ks} gives the bound ${\rm sdepth}(S/I) \geq 4 - 3 = 1$ and ${\rm sdepth}(I) \geq 4 - 3+1 = 2$. On the other hand, one can easily see that ${\rm sv}(I)=2$. Thus, Theorem \ref{smain} implies that ${\rm sdepth}(S/I) \geq 2$ and ${\rm sdepth}(I) \geq 3$. We note that in \cite[Thorem 3.3]{s2}, the author determines a lower bound for the Stanley depth of squarefree monomial ideals. But this bound is strengthened by \cite[Corollary 2.5]{ks} (see also \cite[Corollary 2.6]{ks}).

\item Let $I \subset S = \mathbb{K}[x_1, \dotsc, x_5]$ be the ideal generated by all squarefree monomials of degree $3$.
	As mentioned in \cite[Examples 3.4]{ks}, the order dimension of $I$ (see \cite[Definitions 1.5]{ks}) is equal to $4$. Thus \cite[Theorem 3.2]{ks} gives the bounds ${\rm sdepth}(S/I) \geq 5 - 4 = 1$ and ${\rm sdepth}(I) \geq 5-4+1=2$. But ${\rm sv}(I)=3$. Indeed, one can consider the following subsets of ${\rm Mon}(I)$: $P_1=\{x_1x_2x_3\}, P_2=\{x_1x_2x_4, x_1x_3x_4, x_2x_3x_4\}$ and $P_3=\{x_1x_2x_5, x_1x_3x_5, x_1x_4x_5, x_2x_3x_5, x_2x_4x_5, x_3x_4x_5\}$. Thus Theorem \ref{smain} implies that ${\rm sdepth}(S/I) \geq 2$ and ${\rm sdepth}(I) \geq 3$.
\end{enumerate}
\end{exmps}



\end{document}